\pdfoutput=1

\documentclass[11pt]{article}

\makeatletter
\renewcommand{\maketitle}{\bgroup\setlength{\parindent}{0pt}
\begin{flushleft}
  \textbf{\@title}

  \@author
\end{flushleft}\egroup
} 
\def\footnoterule{\kern-3\p@
  \hrule \@width 2in \kern 2.6\p@} 
\makeatother

\usepackage{preamble}
\usepackage{courier}
\usepackage[square,numbers]{natbib}

\title{{\Large \textbf{On f-generic types in NIP groups}}}
\author{Atticus Stonestrom\ \ \ \ \ \ \ \ \ \ \ \ \ \ \ \ \ \  {\footnotesize (Email: \texttt{atticusstonestrom@yahoo.com})}}
\date{}

\begin{document}
\maketitle

{\small\noindent\textbf{Abstract:} `Definable amenability' of a definable group is the model-theoretic analogue of amenability of a discrete group; precisely, a definable group is said to be definably amenable if it admits a translation-invariant finitely additive probability measure on its definable subsets. We prove a combinatorial characterization of definable amenability for groups definable in NIP theories. More specifically, given a group $G$, a subset $D\subseteq G$ is said to (left) `$G$-divide' if there is some natural number $k$ and an infinite sequence of elements $g_i\in G$ such that $g_{i_1}D\cap\dots\cap g_{i_k}D=\varnothing$ for all $i_1<\dots<i_k$. Our main result is that, if $G$ is a group definable in an NIP theory, and the union of two definable $G$-dividing subsets of $G$ still $G$-divides, then $G$ is definably amenable. It follows that $G$ is definably amenable if and only if $G$ admits a global non-$G$-dividing (or, equivalently, `f-generic') type. This answers a question of Chernikov and Simon and generalizes a theorem of Hrushovski and Pillay. As a quick application of the main result, we show that every dp-minimal group is definably amenable, which answers a question of Chernikov, Pillay, and Simon. Finally, we show that the appropriate analogue of the main result holds also for type-definable groups, so that, in an NIP theory, a type-definable group with a global f-generic type is definably amenable; this additionally gives the first correct proof of the analogous result, claimed by Hrushovski and Pillay, for type-definable groups with a global \textit{strongly} f-generic type. \newline

\noindent\textbf{Acknowledgements:} I would especially like to thank both Artem Chernikov and my PhD advisor Anand Pillay for their extensive encouragement and support. I would additionally like to thank Itay Kaplan, Krzysztof Krupi\'{n}ski, and Pierre Simon for taking the time to read earlier versions of this paper. I would like to thank Itay Kaplan in particular for pointing out a shorter proof of Lemma \ref{sms lemma}. Finally, I would like to thank the anonymous referees for numerous advices and comments to improve the presentation of the paper. \newline

\noindent\textbf{AI Declaration:} AI was not used for any part of this paper, neither in the mathematical content nor in the writing.\newline

\noindent\textbf{Notation:} Throughout, $T$ will denote a complete theory in a first-order language $L$, with a monster model $\mathfrak{C}$. Everywhere except Section \ref{type-definable section}, $G$ will denote a $\varnothing$-definable group in $T$ (in Section \ref{type-definable section}, it will be a type-definable group), and for any model $M\preccurlyeq\mathfrak{C}$ we write $G(M)$ for the $M$-points of $G$; as usual we will often notationally identify $G$ with $G(\mathfrak{C})$ when it does not cause confusion. For a parameter set $C\subseteq\mathfrak{C}$, we use $S_G(C)$ to denote the space of all complete types over $C$ that are concentrated on $G$. For tuples $a,b$, and a small parameter set $C\subset\mathfrak{C}$, we write $a\equiv_C b$ to mean $\tp(a/C)=\tp(b/C)$. We write $\ind$ for non-forking independence. We write $[n]$ to mean $\{1,\dots,n\}$, so that in particular $[0]=\varnothing$. Finally, everywhere except Section \ref{dp-minimal section}, we will assume that $T$ is NIP.

} 

\section{Introduction}
The study of groups definable in first-order theories has long been of interest in model theory, and a particularly successful and important example has been the study of `stable groups', i.e. groups definable in stable theories. (Recall that a theory is `stable' if there is no definable relation which induces a linear order on an infinite set.) Examples of stable groups include any abelian group in the pure group language, any algebraic group in an algebraically closed field, and any differential algebraic group in a differentially closed field. In general, there is a robust machinery available for analyzing the structure of definable groups in arbitrary stable theories (see e.g. \cite{wagner}), and this structure theory and appropriate analogues of it have played a key role in many of the major applications of model theory to algebraic geometry, differential algebraic geometry, and additive combinatorics.

In the attempts to generalize the success of stability theory to more general settings, it is thus natural, in particular, to attempt to generalize the structure theory of stable groups to more general families of definable groups. This paper deals specifically with `NIP groups,' i.e. groups definable in NIP theories. (Recall that a theory is `NIP' if every uniformly definable family of sets has finite VC-dimension – equivalently, if there is no definable graph that contains an isomorphic copy of every finite graph as an induced subgraph.) Stable structures are all NIP, but the class of NIP structures is much broader, and includes for instance the real field (or more generally any o-minimal structure), the $p$-adic fields, the ordered additive group of integers, and algebraically closed valued fields. Correspondingly, the class of NIP groups is also much broader than that of stable groups, and includes any real or $p$-adic semialgebraic group, as well as other Lie groups such as the universal cover of $\mathrm{SL}_2(\mathbb{R})$. \cite{hpp_2}

A productive and important way of viewing NIP groups, originally proposed by Newelski in \cite{newelski}, is via topological dynamics. In the most naive version, given a structure $M$ and a group $G$ definable in $M$, one considers the continuous action of the discrete group $G(M)$ on the compact type space $S_G(M)$, where elements of $G(M)$ act by left-translation on types in $S_G(M)$. This yields a $G(M)$-flow; if the structure $M$ is NIP, then the $G(M)$-flow is `tame' in the sense of \cite{glasner}: namely, for any continuous real-valued map $f\in C(S_G(M))$, and for any sequence $g_i\in G(M),i\in\omega$, the sequence $(f\circ g_i:i\in\omega)$ in the Banach space $(C(S_G(M)),||\cdot||_\mathrm{sup})$ is not an `$\ell^1$-sequence' in the sense of \cite{rosenthal}; see e.g. Section 4 of \cite{krupinski_rzepecki}. So, one can think of the NIP condition on the definable group $G$ as a model-theoretic condition that guarantees the tameness of its associated flow, and our main result can be understood as a statement about this tame flow, in a way we will make precise below.

In the attempt to study NIP groups, one approach has been to restrict the class of NIP theories one is interested in and study the definable groups there; an extremely successful example has been the classification of definable groups in o-minimal theories, like that of the real field. Another example of this nature is the study of definable groups specifically in the $p$-adic fields, which is a very active area with some significant open problems. However, just as one has for arbitrary stable theories, one also wants a structure theory for definable groups in \textit{arbitrary} NIP theories.

Thus far, it has only been possible to do this after placing additional restrictions on the definable groups one considers, and the most general assumption which has yielded a successful structure theory is that of `definable amenability': a definable group is said to be definably amenable if it admits a translation-invariant finitely additive probability measure on the Boolean algebra of its definable subsets. In the dynamical language above, the definable amenability condition says precisely that there is a $G(M)$-invariant regular Borel probability measure on $S_G(M)$. Definable amenability was first introduced without a name in \cite{newelski_petrykowski}, where it was proven that stable groups are definably amenable, and the notion was then named and studied in its own right in \cite{hrushovski_peterzil_pillay}. Along with stable groups, examples of definably amenable groups include definable groups whose realizations in some model are amenable as a discrete group (so, for instance, expansions of solvable groups), pseudofinite groups, and groups definable in $F=(\mathbb{R},\cdot,+)$ and $F=(\mathbb{Q}_p,\cdot,+)$ whose $F$-points are compact in the inherited topology; non-examples include $\mathrm{SL}_2$ as a definable group in the fields $F$ given above, an example given already in \cite{hrushovski_peterzil_pillay} (see directly after the statement of Theorem 1.1 below for an argument). 

The structure theory of definably amenable NIP groups has been very successfully developed, with major progress made over the course of the three papers \cite{hrushovski_peterzil_pillay}, \cite{hrushovski_pillay}, and \cite{chernikov_simon}. The original motivation for this work arose in the investigation of `Pillay's conjecture' for groups definable in o-minimal structures, but the notion was quickly seen to be of importance more generally and investigated in the general NIP setting. A key part of this investigation was to find an appropriate analogue for the theory of generic sets and types in stable groups, which we recall quickly now

A subset of a group is called (left) `generic' if finitely many left translates of it cover the group. Genericity is a well-behaved notion in stable groups: the non-generic definable subsets form an ideal, allowing for the construction of global `generic types' (global types that do not concentrate on any non-generic definable subset), which play an essential role in stable group theory.

Generic types need not exist in unstable definable groups; for example, the underlying group of $(\mathbb{R},+,<)$ does not admit generic types, since neither $(-\infty,0]$ nor $[0,\infty)$ is a generic subset. However, for definably amenable NIP groups, the `correct' generalization of genericity is given in \cite{chernikov_simon}: in a group $G$, a subset $D\subseteq G$ is said to (left) `$G$-divide' if and only if there is some $k\in\omega$ and some sequence $g_i\in G,i\in\omega$ such that $\bigcap_{i\in s}g_iD=\varnothing$ for every $s\subset\omega$ of size $k$. The relevant definable subsets of $G$ in the definably amenable NIP setting are then those sets that do not $G$-divide, a notion that is extensively investigated and employed in \cite{chernikov_simon}.

A definable subset of a group $G$ definable in an NIP theory does not $G$-divide if and only if it is `f-generic' (see Section \ref{f-genericity section} for the definition). Moreover, by a Ramsey's theorem argument, a definable subset of $G$ that $G$-divides must have measure $0$ under any translation-invariant Keisler measure on $G$. In \cite{chernikov_simon} it is shown that, if $G$ is definably amenable and NIP, then the non-f-generic definable subsets of $G$ form a translation-invariant ideal, and the corresponding machinery of f-generic types is then used to answer a swathe of open problems on definably amenable NIP groups. Moreover the ideal of non-f-generic sets was shown there to be canonical in the following sense: in a definably amenable NIP group, the non-f-generic definable sets are precisely the definable sets which get measure $0$ under any translation-invariant Keisler measure.

However, an important gap in the theory was left open, and posed as Question 3.18 in \cite{chernikov_simon}: whether, in an NIP group, the existence of the ideal of non-f-generic definable sets – or equivalently (see Section \ref{f-genericity section}) the existence of a global f-generic type – implies (and is hence equivalent to) definable amenability. In this paper we give a positive answer:

\begin{theorem}
    Suppose $G$ is a group definable in an NIP theory. If $G$ admits a global f-generic type, then $G$ is definably amenable.
\end{theorem} This is given by Theorem \ref{DA main theorem}. For a illustration of the kind of behavior described by theorem, let us show why $\mathrm{SL}_2$ is not a definably amenable group in the structure $(\mathbb{R},+,\cdot)$. Consider the definable subset $D=\left\{\begin{bmatrix} a & b \\ c & d \end{bmatrix}:|a|\geqslant|c|\right\}$ of $G=\mathrm{SL}_2(\mathbb{R})$. On the one hand, letting $g_n=\begin{bmatrix} 1 & 0 \\ 3n & 1 \end{bmatrix}$  for each $n\in\omega$, then we have that $g_iD\cap g_jD=\varnothing$ for each $i\neq j$, so that $D$ $G$-divides. On the other hand, by a similar argument using the transpose $g_n^{\mathrm{tr}}$, the complement of $D$ also $G$-divides. So $\mathrm{SL}_2(\mathbb{R})$ is a union of two definable sets which must have measure $0$ under any left-translation-invariant probability measure, and hence there can be no such measure. Our main result is that every failure of definable amenability for an NIP group arises from a similar combinatorial situation.

The result is a significant generalization of an earlier result from \cite{hrushovski_pillay}, which showed that, if $G$ is a group definable in a countable\footnote[1]{In \cite{hrushovski_pillay} this result is claimed for arbitrary NIP theories, not just countable ones. But a gap in the proof for the uncountable case was found by Krupi\'{n}ski; see Section \ref{strong f-genericity subsection} for more discussion on this.} NIP theory, and $G$ admits a \textit{strongly} f-generic type (see Section \ref{strong f-genericity subsection} for the definition), then $G$ is definably amenable. The notion of strong f-genericity was the first attempt to find the right generalization of genericity to the definably amenable NIP setting; in \cite{hrushovski_pillay}, following nomenclature from \cite{newelski_petrykowski}, strong f-genericity is just called f-genericity, but the notion was renamed in \cite{chernikov_simon} after realizing that f-genericity as defined in this paper is the more suitable and natural generalization. In particular, f-genericity is substantially more general than strong f-genericity, and even in some basic examples of definably amenable NIP groups (such as the group $(\mathbb{R}^2,+)$ in the structure $(\mathbb{R},+,\cdot)$) there are unboundedly many global f-generic types that are not strongly f-generic; see Example 3.10 in \cite{chernikov_simon}. Strong f-genericity does not have the same kind of combinatorial definition (purely in terms of the group operation) that genericity or f-genericity (in the form of $G$-dividing) have, which also prevents it from having something like the natural dynamical interpretation of $G$-dividing that we point out below. After their extensive role in \cite{chernikov_simon} in the resolution of a number of open problems, f-generic types took on a central importance in the theory of definably amenable NIP groups; as such, the inability to show that their existence is enough to imply definable amenability was a major gap in the structure theory, which this paper now fills.

Altogether, in conjunction with results from \cite{hrushovski_pillay} and \cite{chernikov_simon}, our main result gives the following:

\begin{theorem}
    The following are equivalent for a group $G$ definable in an NIP theory:
\begin{enumerate}
        \item $G$ is definably amenable.
        \item $G$ admits a global type invariant under left-translation by the elements of $G^{00}$.
        \item $G$ admits a global f-generic type.
        \item The non-f-generic, equivalently $G$-dividing, definable subsets of $G$ form an ideal.
\end{enumerate}\end{theorem} (This is Corollary \ref{DA main corollary}.)

In the dynamical language discussed above, our main result can be viewed as a statement about the flow $(G(\mathfrak{C}),S_G(\mathfrak{C}))$ as follows. Assume that $\mathfrak{C}$ is NIP, so that $(G(\mathfrak{C}),S_G(\mathfrak{C}))$ is tame, and say that a $\{0,1\}$-valued function $f\in C(S_G(\mathfrak{C}))$ `$G$-divides' if and only if there is $k\in\omega$ and a sequence $(g_i:i\in\omega)$ from $G(\mathfrak{C})$ such that, for any $s\subset\omega$ of size $k$, we have $\prod_{i\in s}(f\circ g_i)=0$. Then our result says that there is a $G(\mathfrak{C})$-invariant regular Borel probability measure on $S_G(\mathfrak{C})$ if and only if, for any continuous $\{0,1\}$-valued functions $f_1,f_2\in C(S_G(\mathfrak{C}))$, if $f_1,f_2$ both $G$-divide then so does $\max\{f_1,f_2\}$. It would be tempting to find an appropriate analogue in more general tame flows – namely, given a tame flow $(H,X)$, to give an appropriate definition of `$H$-dividing' for an arbitrary continuous function $f\in C(X)$, not just a $\{0,1\}$-valued one, and to investigate the existence of $H$-invariant regular Borel probability measures on $X$ through that lens. That falls outside the scope of this paper, but is something the author hopes to pursue in the future. (A specific case would be the action of definable groups on their types spaces in \textit{continuous} NIP theories; in this setting one expects an appropriate analogue of Theorem 1.2 to go through. A study of definably amenable groups in continuous logic is done in \cite{carmona_onshuus}, where in particular the analogue of the result on \textit{strong} f-generic types from \cite{hrushovski_pillay} is shown for continuous NIP theories.)

As a model-theoretic application of our main theorem, we prove a result on dp-minimal groups. `Dp-minimality' is a kind of abstract one-dimensionality condition for NIP theories that generalizes o-minimality and strong minimality, which was first studied in \cite{shelah_forking} and \cite{shelah_strong} and then isolated as a notion in \cite{onshuus_usvyatsov}. It was asked in \cite{chernikov_pillay_simon} and \cite{kaplan_levi_simon} whether every dp-minimal group is definably amenable. We give a positive answer here, obtained as a quick corollary of the results from Section \ref{results section}:
\begin{theorem}
    If $G$ is a dp-minimal definable group, then $G$ is definably amenable.
\end{theorem} This is given by Theorem \ref{dp-minimal main theorem}. In later work, this result is used as a starting point to show that dp-minimal groups are nilpotent-by-finite; see \cite{stonestrom} and \cite{alouf_stonestrom_wagner}.

Finally, in Section \ref{type-definable section}, we treat the case of `type-definable' groups, namely groups that are given as the intersection of a (possibly infinite) family of definable sets. In the stable setting, it is a theorem of Hrushovski that every type-definable group is an intersection of definable groups, but this need not be true in the NIP setting – for example, the group of `infinitesimals' of $(\mathbb{R},+,\cdot)$, i.e. the elements of a saturated elementary extension that satisfy $-1/n<x<1/n$ for every $n\in\mathbb{N}$, will be a type-definable subgroup of the additive group that is not an intersection of definable groups. Nevertheless, model-theoretic theorems about definable groups often hold in the more general type-definable setting, albeit at times with a bit more care required in the proof. We show in Section \ref{type-definable section} that this is the case for our main result: \begin{theorem}
    Let $G$ be a type-definable group in an NIP theory. If there is a global type concentrated on $G$ that is f-generic for $G$, then $G$ is definably amenable.
\end{theorem} This is Theorem \ref{main theorem type-definable}. This result also gives the first correct proof, to our knowledge, of Proposition 5.7 in \cite{hrushovski_pillay}, which claims that a \textit{type-definable} group in an NIP theory admitting a global \textit{strongly} f-generic type is definably amenable; see Section \ref{strong f-genericity subsection} and Section \ref{type-definable section} for more discussion of this.

\section{Preliminaries}\label{preliminaries}
As in the rest of the paper, we assume throughout this section that $T$ is NIP.

\subsection{Forking in NIP Theories}\label{forking section}
We recall here some properties of non-forking independence in NIP theories. In this paper we will always be working over models, so we state all the necessary facts only for that context. Recall that a formula $\phi(x,b)$ with parameters from $\mathfrak{C}$ `divides' over a parameter set $A\subset\mathfrak{C}$ if there is some $k\in\omega$ and some sequence $(b_i:i\in\omega)$ such that $b_i\equiv_A b$ for all $i\in\omega$ and such that the family of formulas $(\phi(x,b_i):i\in\omega)$ is $k$-inconsistent. A partial type divides over $A$ if it implies a formula that divides over $A$, and it `forks' over $A$ if it implies a finite disjunction of formulas that each divide over $A$. We write $a\ind_A b$ to mean that $\tp(a/A,b)$ does not fork over $A$.

Recall that a global type $p(x)\in S(\mathfrak{C})$ is said to be `$M$-invariant' over some small model $M\prec\mathfrak{C}$ if, for any tuples $b,b'$ from $\mathfrak{C}$ with $b\equiv_M b'$, we have $p(x)\vdash\phi(x,b)$ if and only if $p(x)\vdash\phi(x,b')$. From \cite{shelah_forking} we first have the following:
\begin{fact}\label{nf iff invariant}
    If $p(x)\in S(\mathfrak{C})$ is a global type and $M\prec\mathfrak{C}$ is a small model, then $p(x)$ does not fork over $M$ if and only if it is $M$-invariant.
\end{fact}

For partial types or complete types that are not global, the relationship of forking and dividing in NIP theories is greatly clarified by the results from \cite{chernikov_kaplan} on the more general class of NTP$_2$ theories. We will rely on a number of those results, which we record in the facts below.

\begin{fact}\label{forking=dividing}
    Let $M\prec\mathfrak{C}$ be a small model. Then an $L(\mathfrak{C})$-formula forks over $M$ if and only if it divides over $M$.
\end{fact}

\begin{fact}\label{left extension}
    Let $M\prec\mathfrak{C}$ be a small model. Then $\ind$ has `left-extension' over $M$: if $A,C$ are sets with $A\ind_M C$, and $B\supseteq A$, then there is $B'$ with $B'\equiv_{M,A} B$ and $B'\ind_M C$.
\end{fact}

\begin{fact}\label{existence of strict non-forking extension}
    Let $p(x)\in S(M)$ be a complete type over a small model $M$. Then there is $q(x)\in S(\mathfrak{C})$ a global extension of $p(x)$ strictly non-forking over $M$. (So, for any small parameter set $C\supseteq M$, if $a\models q|_C$ then $a\ind_M C$ and $C\ind_M a$.)
\end{fact}Finally, we also recall that $\ind$ in arbitrary theories has `left-transitivity' (see \cite{adler_forking}): for any small sets $A,B,C,D\subset\mathfrak{C}$, if $A\ind_{(B,C)} D$ and $B\ind_C D$ then $(A,B)\ind_C D$.

\subsection{Connected Components}\label{connected components section} Recall that, for a small model $M\prec\mathfrak{C}$, $G^{00}_M$ denotes the smallest subgroup of $G$ type-definable over $M$ and of bounded index, and $G^{\infty}_M$ denotes the smallest subgroup of $G$ `invariant over $M$', i.e. setwise invariant under automorphisms fixing $M$ pointwise, and of bounded index; one can think of $G^{\infty}_M$ (resp. $G^{00}_M$) as the smallest subset (resp. smallest closed subset) of $S_G(M)$ whose set of realizations $G^{\infty}_M(\mathfrak{C})$ (resp. $G^{00}_M(\mathfrak{C})$) is a subgroup of $G(\mathfrak{C})$ of index smaller than the degree of saturation of $\mathfrak{C}$. An explicit description of $G^{\infty}_M(\mathfrak{C})$ is as the subgroup of $G(\mathfrak{C})$ generated by $\{a^{-1}b:a,b\in G(\mathfrak{C}),a\equiv_M b\}$; see for example 8.1.4 in \cite{simon_book} for a reference.

The groups $G^{00}_M$ and $G^\infty_M$ are normal subgroups of $G$, and one can endow the quotient $G/G^{00}_M$ with the `logic topology', where a subset is closed if and only if its preimage under the projection map $G\to G/G^{00}_M$ is type-definable over some small set. This makes $G/G^{00}_M$ into a compact Hausdorff topological group, which is thus endowed with a (unique) translation-invariant Haar measure $h$ such that $h(G/G^{00}_M)=1$.

When $G^{00}_M$, respectively $G^{\infty}_M$, is independent of the choice of $M$, one says that $G^{00}$, respectively $G^{\infty}$, `exists' and drops the subscript. From \cite{shelah_g00} and \cite{gismatullin} respectively, it is known that $G^{00}$ and $G^{\infty}$ always exist if $T$ is NIP.

For $T$ countable and NIP, Hrushovski and Pillay gave a construction in \cite{hrushovski_pillay} to obtain translation-invariant Keisler measures on $G$ from $G^{00}$-invariant types:

\begin{definition}\label{construction of measure}
Suppose $p(x)\in S_G(\mathfrak{C})$ is invariant under left translation by elements of $G^{00}(\mathfrak{C})$. For a $\mathfrak{C}$-definable set $D\subseteq G$, define $S_{p,D}$ to be the subset of $G/G^{00}$ given by the set of cosets $\{gG^{00}:gp(x)\vdash x\in D\}$; this is well-defined by $G^{00}$-invariance of $p$. If the sets $S_{p,D}$ are all Borel, then we define a Keisler measure $\mu_p$ on $G$ by taking $\mu_p(D)=h(S_{p,D})$ for each $D$. Note that $\mu_p$ will be left-invariant by left-invariance of $h$.
\end{definition}

In the case that $T$ is countable and NIP, the sets $S_{p,D}$ given in Definition \ref{construction of measure} will indeed all be Borel. This is proved in \cite{hrushovski_pillay} under the additional hypothesis that $p(x)$ is $M$-invariant over some countable $M\prec\mathfrak{C}$, and a further argument of Chernikov and Simon shows that this hypothesis is not needed; see Definition 3.16 of \cite{chernikov_simon}.\footnote[1]{In \cite{chernikov_simon} this is not expressed that way; Chernikov and Simon assume there that $G$ is definably amenable and that $p(x)$ is f-generic. But all that one needs to apply their argument is that what they denote $p_M$ is $M$-invariant for every small $M\prec\mathfrak{C}$, and this is true whenever $p$ is $G^{00}$-invariant: for $p_M$ to not be $M$-invariant means precisely that $p(x)\vdash x\in aD\setminus bD$ for some $M$-definable $D\subseteq G$ and some $a\equiv_M b$ in $G$, but then then $ab^{-1}\in G^{00}$ and $ab^{-1}p(x)\vdash x\notin aD$, so that $p$ is not $G^{00}$-invariant.} Thus one has the following fact:

\begin{fact}\label{countable case}
    If $T$ is countable and NIP, and $G$ admits a global $G^{00}$-invariant type, then the group $G$ is definably amenable.
\end{fact}

We will not need it here, but it was also proved in \cite{hrushovski_pillay} that a definably amenable NIP group admits a global $G^{00}$-invariant type, so Fact \ref{countable case} is in fact an equivalence.

As a brief remark, it is worth mentioning `Petrykowski's conjecture', as in \cite{newelski_bdd}, which says that, for an arbitrary definable group $G$, not necessarily NIP, if there is an orbit of bounded size in the flow $(G(\mathfrak{C}),S_G(\mathfrak{C}))$, then $G$ is definably amenable. The aforementioned Theorem 3.12 of \cite{chernikov_simon} positively resolved Petrykowski's conjecture in the NIP case. The connection with Fact \ref{countable case} is that a global $G^{00}$-invariant type will have bounded orbit, corresponding to the boundedly many cosets of $G^{00}$ in $G$. Without the NIP assumption, Petrykowski's conjecture remains an important open problem in the study of definably amenable groups.

\subsection{f-Genericity}\label{f-genericity section}
\subsubsection{General facts}\label{general facts subsection}
Here we record a few facts about f-generic formulas from \cite{chernikov_simon}; throughout this section assume $T$ is NIP. In fact, the more appropriate setting for everything in this section is that of NTP$_2$ theories, a common generalization of NIP and simple theories, and everything here is true in the NTP$_2$ case, as proved in Section 3 of \cite{montenegro_onshuus_simon}. But we wish to focus just on the NIP setting for clarity. So we continue assuming $T$ is NIP.

We define f-generic following \cite{chernikov_simon}:
\begin{definition}
    A $\mathfrak{C}$-definable set $D\subseteq G$ is `f-generic' if, for any small model $M$ over which $D$ is defined, the formula $x\in gD$ does not fork over $M$ for all $g\in G$.\footnote[2]{By a slightly abuse of terminology, given an $M$-definable set $D$, we will often say that a translate $gD$ `does not fork over $M$' to mean that the formula $x\in gD$ does not fork over $M$. Since $D$ is defined over $M$, and forking and dividing coincide over $M$ in our context, this should not cause confusion.} Likewise, a partial type is called f-generic if it implies only f-generic formulas.
\end{definition} Note immediately that, given a model $M$, a partial type over $M$ is f-generic if and only if no translate of it forks over $M$. Also, by Fact \ref{forking=dividing} and Ramsey's theorem, f-genericity (like genericity) can be characterized purely in terms of the group structure on $G$:

\begin{fact}\label{G-dividing characterization}
    A $\mathfrak{C}$-definable set $D\subseteq G$ is not f-generic if and only if it `$G$-divides', i.e. if and only if there is some $k\in\omega$ and some sequence $(g_i)_{i\in\omega}$ from $G$ such that the family of translates $(g_iD)_{i\in\omega}$ is $k$-inconsistent.
\end{fact} Throughout this paper we will freely use the equivalence between $G$-dividing and non-f-genericity, typically without mention. Now, the following fact is a strengthening of Proposition 3.4 of \cite{chernikov_simon}; the proof is exactly the same as Proposition 3.4 there, but the result there is not stated as such, so we include the proof here for completeness.

\begin{proposition}\label{universal witnesses}
    Suppose $D\subseteq G$ is definable over a small model $M$, and that $g\in G(\mathfrak{C})$ is such that $r(x):=\tp(g/M)$ is f-generic. If $g^{-1}D$ does not fork over $M$, then $D$ is f-generic.
\end{proposition}
\begin{proof}
    Suppose $D$ is not f-generic. Then there is an $M$-indiscernible sequence $J=(h_i)_{i\in\omega}$ such that $\bigwedge_{i\in\omega}h_iD=\varnothing$; let $k$ be such that the translates $h_iD$ are $k$-inconsistent. Since $r(x)$ is a partial type over $M$ and is f-generic, the translate $h_0r(x)$ does not fork over $M$. In particular, since $J$ is $M$-indiscernible, the partial type $\bigwedge_{i\in\omega}h_ir(x)$ is consistent. Let $u\in G(\mathfrak{C})$ be any realization, and let $g_i=h_i^{-1}u$ for each $i\in\omega$. Then $g_i\models r(x)$, so that $g_i^{-1}\equiv_M g^{-1}$, for each $i\in\omega$. But the translates $g_i^{-1}D$ are $k$-inconsistent; indeed, for any $s\subset\omega$ of size $k$, we have $$\bigwedge_{i\in s}g_i^{-1}D=u^{-1}\bigwedge_{i\in s}h_iD=\varnothing.$$ Thus $g^{-1}D$ divides over $M$.
\end{proof} One has the following standard consequence; the proof is identical to that of Corollary 3.5 in \cite{chernikov_simon}, but we include it for completeness.

\begin{corollary}\label{global type iff ideal}
    $G$ admits a global f-generic type if and only if the non-f-generic definable subsets of $G$ form an ideal.
\end{corollary}
\begin{proof}
    On the one hand, if the non-f-generic definable subsets form an ideal, then in particular $G$ itself is not a union of finitely many non-f-generic definable subsets. So the collection of formulas $\{\phi(x,c)\in L(\mathfrak{C}):\neg\phi(x,c)\text{ is not f-generic}\}$ is finitely consistent, and any complete global type extending it will be f-generic.

    On the other hand, suppose $G$ admits a global f-generic type $p(x)\in S_G(\mathfrak{C})$, and let $D,E\subseteq\mathfrak{C}$ be non-f-generic definable subsets. Let $M$ be a small model over which both $D,E$ are defined, and let $g\models p|_M$. Then, by Proposition \ref{universal witnesses}, both $g^{-1}D$ and $g^{-1}E$ fork over $M$. So in particular $g^{-1}(D\vee E)=g^{-1}D\vee g^{-1}E$ forks over $M$, whence (since $D\vee E$ is defined over $M$) $D\vee E$ is not f-generic.
\end{proof}

In \cite{chernikov_simon} it is shown that, in a \textit{definably amenable} NIP group, a global type is f-generic if and only if it is $G^{00}$-invariant. The proof of the backwards direction does not require definably amenability; indeed, suppose $p(x)\in S_G(\mathfrak{C})$ is not f-generic, and let $D$ be a non-f-generic definable set concentrated on by $p(x)$. Then in particular there is a sequence $(g_i:i\in\omega)$ of elements of $G(\mathfrak{C})$, indiscernible over the parameters defining $D$, such that the conjunction $\bigwedge_{i\in\omega}g_iD$ is inconsistent. So $D\wedge\bigwedge_{i>0}g_0^{-1}g_iD$ is also inconsistent, whence $p(x)$ must not concentrate on some $g_0^{-1}g_iD$. But each $g_0^{-1}g_i$ lies in $G^{00}(\mathfrak{C})$, so that $p(x)$ is not $G^{00}$-invariant.

On the other hand, the proof in \cite{chernikov_simon} of the forwards direction, that a global f-generic type is $G^{00}$-invariant, relies heavily on the definable amenability hypothesis. In fact, the main technical result of our paper here is to show that the forwards direction also holds even without a definable amenability assumption; this is Corollary \ref{f-generic is g00-invariant}.

\subsubsection{Strong f-genericity}\label{strong f-genericity subsection}
For completeness, let us now discuss strong f-genericity, although it is not a necessary notion for our paper. None of the material in this section is necessary for the results of our paper, but, in order to contextualize our results, it is perhaps worth remarking on how the various notions connect. Following the terminology of \cite{chernikov_simon}, we say that a global type $p(x)\in S_G(\mathfrak{C})$ is (left) `strongly f-generic' if there is some small model $M\prec\mathfrak{C}$ such that no left translate of $p(x)$ forks over $M$; in other words, a global type $p(x)\in S_G(\mathfrak{C})$ is strongly f-generic if and only if there is some small model $M\prec\mathfrak{C}$ such that, for every $\phi(x,b)\in p$, no left translate of $\phi(x,b)$ forks over $M$. In contrast, a global type $p(x)\in S_G(\mathfrak{C})$ is f-generic if and only if, for every $\phi(x,b)\in p$, there is some small model $M\prec\mathfrak{C}$ such that no left translate of $\phi(x,b)$ forks over $M$. This swap of quantifiers in the definition is an essential change, and even in some very basic examples of definably amenable NIP groups, such as $(\mathbb{R}^2,+)$ in $(\mathbb{R},+,\cdot)$, there are unboundedly many types that are f-generic but not strongly f-generic; see Example 3.10 in \cite{chernikov_simon}.\footnote[1]{On the other hand, for groups definable in simple theories, the analogous notions of f-genericity and strong f-genericity coincide, by the results in \cite{pillay_simple}.}

In \textit{definably amenable} NIP groups, the relationship between f-genericity and strong f-genericity is described in \cite{chernikov_simon}. To see it, first note the following easy observation, which is perhaps worth recording: 

\begin{proposition}
Let $G$ be an arbitrary definable group in an arbitrary theory, not necessarily NIP. Suppose that $p(x)\in S_G(\mathfrak{C})$ has bounded orbit under the action of $G(\mathfrak{C})$. Then $p(x)$ is strongly f-generic if and only if there is some small model $M\prec\mathfrak{C}$ such that $p(x)$ does not fork over $M$.
\end{proposition}
\begin{proof}
    For the non-trivial direction, suppose that $p(x)$ does not fork over some small $M\prec\mathfrak{C}$. Since $p(x)$ has bounded orbit under the action of $G(\mathfrak{C})$, we may find elements $\{g_i:i\in I\}$ of $G(\mathfrak{C})$ such that $|I|$ is small and such that, for every $g\in G(\mathfrak{C})$, there is some $i\in I$ with $gp=g_ip$. Letting $N$ be any small model containing $M$ and $\{g_i:i\in I\}$, then no left translate of $p$ forks over $N$.
\end{proof} As remarked at the end of Section \ref{general facts subsection}, it is shown in \cite{chernikov_simon} that a global f-generic type in a \textit{definably amenable} NIP group $G$ is $G^{00}$-invariant, and hence has bounded orbit. So it follows, as proved in \cite{chernikov_simon}, that, in a \textit{definably amenable} NIP group, a global type is strongly f-generic if and only if it is f-generic and non-forking over some small model. However, this equivalence is not clear without the definable amenability assumption. It is easy to show that the existence of a global f-generic type implies the existence of a global f-generic type that is non-forking over some small model; see Lemma \ref{existence of invariant f-generic} below. The hard part is to show that the f-generic type in question has bounded orbit, and, again as mentioned above, that is in fact the main technical result of our paper.

On the other hand, in Proposition 5.6(i) of \cite{hrushovski_pillay}, it is shown that a strongly f-generic type in an NIP group $G$ is $G^{00}$-invariant. Using the argument of Fact \ref{countable case} cited above, it was thus shown in Proposition 5.6(ii) of \cite{hrushovski_pillay} that a definable group in a countable NIP theory which admits a strongly f-generic type is definably amenable.

In fact, Proposition 5.6(ii) of \cite{hrushovski_pillay} claims to show this result even without the countability hypothesis, the idea being that a definable group is definably amenable if and only if it is definably amenable in every reduct of the theory to a countable sublanguage over which it is still defined. However, a key gap in this argument was found by Krupi\'{n}ski, which is that it is non-obvious that a strongly f-generic type will remain strongly f-generic in a reduct. A solution to this problem, which uses Theorem 3.12 in \cite{chernikov_simon}, was later found by Krupi\'{n}ski and Pillay; for the proof we refer the reader to Proposition 3.16 in \cite{pillay_groups}. The upshot is that, if $G$ is a definable group in an NIP theory $T$, and $G$ admits a global strongly f-generic type, then $G$ is definably amenable.

In Remark 5.7 of \cite{hrushovski_pillay}, it is claimed that this result also holds when $G$ is just type-definable. If $T$ is countable, and $G$ is type-definable over a countable parameter set, then this indeed follows by the same proofs of Proposition 5.6(i) and Proposition 5.6(ii) in \cite{hrushovski_pillay}. More generally, if $T$ is not necessarily countable, but $G$ is type-definable by an intersection of countably many formulas, then the result also holds, and can be proved by adapting the arguments of Krupi\'{n}ski and Pillay described in Proposition 3.16 of \cite{pillay_groups}. However, in the case where $G$ is not type-definable by an intersection of countably many formulas, neither of those arguments works, since there is no clear way to reduce to the case of a countable language and a countable parameter set. In Section \ref{type-definable section} we will prove that, if $G$ is type-definable and admits a global f-generic type, then $G$ is definably amenable; this will in particular give the first proof for the general claim made in Remark 5.7 of \cite{hrushovski_pillay}.

\section{Results}\label{results section}
Now we can begin proving the result. As always, we assume throughout that $T$ is NIP and that $G$ is a definable group of $T$.
\subsection{Strict Morley Sequences}
First we need the following general observation.\footnote[1]{Thank you to Itay Kaplan for pointing out a shorter argument for this than I originally had.}
\begin{lemma}\label{sms lemma}
    Let $M$ be a small model, and suppose $a,b\in\mathfrak{C}$ are such that $a\equiv_M b$. Suppose also that $q(x,y)\in S(\mathfrak{C})$ is a global extension of $\tp(a,b/M)$ strictly non-forking over $M$, and that $(a_i,b_i)_{i\in\omega}\models q^{\otimes\omega}|_M$ is a Morley sequence of $q$ over $M$. Then, for every $n\in\omega$, there is a model $N\prec\mathfrak{C}$ containing $(M,a_{\neq n},b_{\neq n})$ and such that $a_n\equiv_N b_n$.
\end{lemma}
\begin{proof} Fix $n\in\omega$. Since $(a_n,b_n)\models q|_{M,a_{<n},b_{<n}}$, we have $(a_{<n},b_{<n})\ind_M (a_n,b_n)$ by strict non-forking of $q$. Moreover, $q^{\otimes\omega}$ is an $M$-invariant type, and $(a_i,b_i)_{i>n}$ realizes its restriction to $(M,a_{\leqslant n},b_{\leqslant n})$, so that $(a_{>n},b_{>n})\ind_M (a_{\leqslant n},b_{\leqslant n})$; in particular $(a_{>n},b_{>n})\ind_{(M,a_{<n},b_{<n})}(a_n,b_n)$. By left-transitivity we thus have $(a_{\neq n},b_{\neq n})\ind_M(a_n,b_n)$, and by Fact \ref{left extension} there is now a model $N$ containing $(M,a_{\neq n},b_{\neq n})$ and such that $N\ind_M(a_n,b_n)$. By Fact \ref{nf iff invariant}, $\tp(N/M,a_n,b_n)$ extends to a global $M$-invariant type, and since $a_n\equiv_M b_n$ this implies that $a_n\equiv_N b_n$, as needed.
\end{proof}

\subsection{f-Generic Types}\label{f-generic types section}
Now we record a few lemmas on f-genericity. First let us make an observation; suppose $D\subseteq G$ is an f-generic set definable over a small model $M$, and that $a,b\in G(\mathfrak{C})$ are elements of $G$ with $a\equiv_M b$. By f-genericity, the formula $x\in aD$ does not fork over $M$, and so is contained in some global $M$-invariant type; this type must then also contain the formula $x\in bD$, and so one concludes that the intersection $aD\wedge bD$ (and hence $D\wedge a^{-1}bD$) is non-empty. If there is a global f-generic type, then this set will in fact also be f-generic. To see this we need the following preliminary\footnote[2]{Note that f-generic types automorphism-invariant over a small model were shown to exist in \cite{hrushovski_pillay} in definably amenable NIP groups, in which they coincide precisely with the `strongly f-generic' types. The point here is to obtain such a type without a definable amenability assumption, just assuming the existence of a global f-generic type.}:

\begin{lemma}\label{existence of invariant f-generic} Let $M$ be a small model. If $G$ admits a global f-generic type, then it admits a global $M$-invariant f-generic type.
\end{lemma}
\begin{proof}
    Let $\pi(x)$ be the partial type containing the formula $x\notin D$ for every $\mathfrak{C}$-definable set $D\subseteq G$ that is not f-generic. By hypothesis, $\pi(x)$ is consistent. It is also $M$-invariant, since the property of being f-generic is preserved under automorphisms. Thus $\pi(x)$ does not divide over $M$. By Fact \ref{forking=dividing}, this means $\pi(x)$ does not fork over $M$. By Fact \ref{nf iff invariant}, $\pi(x)$ thus extends to a global $M$-invariant type, which is then f-generic, as needed.
\end{proof}

Now we can obtain the desired strengthening of the observation above.
\begin{lemma}\label{chain condition}
    Suppose there is a global f-generic type. Then, for any f-generic set $D\subseteq G$ definable over a small model $M$, and any $a,b\in G(\mathfrak{C})$ with $a\equiv_M b$, the intersection $D\wedge a^{-1}bD$ is f-generic.
\end{lemma}
\begin{proof} By Lemma \ref{existence of invariant f-generic}, let $q(x)\in S_G(\mathfrak{C})$ be f-generic and $M$-invariant. As f-genericity is translation-invariant, it suffices to show that $E:=aD\wedge bD$ is f-generic. Let $N$ be any small model containing $(M,a,b)$ and let $g\models q|_N$; then by Proposition \ref{universal witnesses} it suffices to show that $g^{-1}E=g^{-1}aD\wedge g^{-1}bD$ does not fork over $N$. Since $D$ is f-generic, $g^{-1}aD$ does not fork over $M$, and so is contained in some global $M$-invariant type $p(x)$. Moreover, $\tp(g/M,a,b)$ extends to the global $M$-invariant type $q$; since $a\equiv_M b$ this implies $a\equiv_{M,g}b$ and hence $g^{-1}a\equiv_M g^{-1}b$. But $p(x)\vdash x\in g^{-1}aD$, so (since $p$ is $M$-invariant) this implies $p(x)\vdash x\in g^{-1}bD$. So $p$ is concentrated on $g^{-1}E$, which hence does not fork over $M$, and in particular does not fork over $N$, as needed.
\end{proof} As a consequence we get the following key lemma (recall that $[n]=\{1,\dots,n\}$ and that $[0]=\varnothing$):

\begin{corollary}\label{key lemma}
    Let $M$ be a small model, and let $I=(c_i)_{i\in\omega}$ be an $M$-indiscernible sequence of elements of $G(\mathfrak{C})$ with the following property: for every $n\in\omega$, there is a model $N$, containing $(M,c_{\neq n})$, and a pair of elements $a,b\in G(\mathfrak{C})$ with $a\equiv_N b$ and $a^{-1}b=c_n$.

    Suppose also that there is a global f-generic type. Then, for any $M$-definable set $D\subseteq G$, the partial type $\{x\in D\}\cup\{x\notin c_iD:i\in\omega\}$ is not f-generic.
\end{corollary}
\begin{proof}
    Suppose otherwise for contradiction, and let $\pi(x)=\{x\in D\}\cup\{x\notin c_{2i}D:i\in\omega\}$; then $\pi(x)$ is f-generic. We claim that $\pi(x)\wedge\bigwedge_{i\in\omega}x\in c_{2i+1}D$ is f-generic; by compactness it suffices to show that $\sigma_n(x):=\pi(x)\wedge\bigwedge_{i\in[n]}x\in c_{2i-1}D$ is f-generic for each $n\in\omega$, and we prove this by induction on $n$. The base case $n=0$ is by hypothesis, and for the inductive step assume we have shown that $\sigma_n(x)$ is f-generic. Now, $\sigma_n$ is a partial type defined over $(M,(c_{2i})_{i\in\omega},(c_{2i-1})_{i\in[n]})$. In particular, by the hypothesis on $I$, there is a small model $N$ such that $\sigma_n(x)$ is defined over $N$ and such that $c_{2n+1}=a^{-1}b$ for some $a,b\in G(\mathfrak{C})$ with $a\equiv_N b$. Since $\sigma_n(x)$ is f-generic, by Lemma \ref{chain condition} and compactness the intersection $\sigma_n(x)\wedge a^{-1}b\sigma_n(x)$, i.e. $\sigma_n(x)\wedge c_{2n+1}\sigma_n(x)$, is thus f-generic. Since $\sigma_n(x)\vdash\pi(x)$ and $\pi(x)\vdash x\in D$, the intersection $\sigma_n(x)\wedge c_{2n+1}\sigma_n(x)$ implies $\sigma_n(x)\wedge x\in c_{2n+1}D=\sigma_{n+1}(x)$, so the result follows.

    So indeed $\pi(x)\wedge\bigwedge_{i\in\omega}x\in c_{2i+1}D$ is f-generic. In particular, it is consistent. But it contains the formulas $x\notin c_{2i}D$ and $x\in c_{2i+1}D$ for every $i\in\omega$; since $I$ is indiscernible this contradicts NIP.
\end{proof}

\subsection{Main Result}\label{main result section}
Now we are ready to prove the main result; we continue to assume that $T$ is NIP.
\begin{lemma}\label{universal s1}
    Let $p(x)\in S_G(\mathfrak{C})$ be a global f-generic type, $M$ a small model, and $D$ an $M$-definable set such that $p(x)\vdash x\in D$. Then the partial type $\{x\in a^{-1}bD:a,b\in G(\mathfrak{C}),a\equiv_M b\}$ is f-generic.
\end{lemma}
\begin{proof}
    Suppose otherwise. Then there are $n\in\omega$ and $(a_1,b_1),\dots,(a_n,b_n)$ such that $a_i\equiv_M b_i$ for each $i$ and such that $\bigwedge_{i\in[n]}a_i^{-1}b_iD$ is not f-generic. By Fact \ref{existence of strict non-forking extension}, let $q(x_1,y_1,\dots,x_n,y_n)$ be a global extension of $\tp(a_1,b_1,\dots,a_n,b_n/M)$ strictly non-forking over $M$, and let $(a_{k1},b_{k1},\dots,a_{kn},b_{kn})_{k\in\omega}\models q^{\otimes\omega}|_M$ be a Morley sequence of $q$ over $M$. In particular, $(a_{k1},\dots,b_{kn})$ has the same type as $(a_1,\dots,b_n)$ over $M$ for every $k\in\omega$, so the set $\bigwedge_{i\in[n]}a_{ki}^{-1}b_{ki}D$ is not f-generic for every $k\in\omega$. Since the type $p$ is f-generic, by the pigeonhole principle there is hence some $i\in[n]$ and some infinite subset $s\subseteq\omega$ with $p(x)\vdash x\notin a_{ki}^{-1}b_{ki}D$ for all $k\in s$.

    Note that the restriction of $q$ to the variables $(x_i,y_i)$ is a global extension of $\tp(a_i,b_i/M)$ strictly non-forking over $M$, and that $(a_{ki},b_{ki}:k\in\omega)$ is a Morley sequence in it. So, by Lemma \ref{sms lemma}, for every $m\in\omega$ there is a small model $N$ containing $(M,a_{ki},b_{ki}:k\neq m)$ and such that $a_{mi}\equiv_N b_{mi}$.
    
    In particular, if we define $c_m=a_{mi}^{-1}b_{mi}$, then $I=(c_m)_{m\in\omega}$ satisfies the hypotheses of Corollary \ref{key lemma}. But now $p$ is f-generic, concentrated on $D$, and contains the formula $x\notin c_kD$ for every $k\in s$; this contradicts Corollary \ref{key lemma}.
\end{proof}

\begin{corollary}\label{g00=g000}
    Suppose there exists a global f-generic type, and let $M$ be a small model. Then the set $\{a^{-1}b:a,b\in G(\mathfrak{C}),a\equiv_M b\}$ is a group, and is hence equal to $G^{00}(\mathfrak{C})$. In particular, $G^{00}=G^{\infty}$.
\end{corollary}
\begin{proof}
    Let $p(x)\in S_G(\mathfrak{C})$ be f-generic. It suffices to show that the set defined above is closed under multiplication, so fix any $a,b,c,d\in G(\mathfrak{C})$ with $a\equiv_M b$ and $c\equiv_M d$. Let $r(x)$ denote the restriction $p|_M(x)$. By Lemma \ref{universal s1} and compactness, the partial type $b^{-1}ar(x)\wedge c^{-1}dr(x)$ is f-generic, and hence in particular consistent. Let $e$ be any realization. Then $a^{-1}be$ and $d^{-1}ce$ each realize $r$, so their inverses $e^{-1}b^{-1}a$ and $e^{-1}c^{-1}d$ have the same type over $M$, and now $a^{-1}bc^{-1}d=(e^{-1}b^{-1}a)^{-1}(e^{-1}c^{-1}d)$.

    So indeed the set in the theorem statement is a group. By the facts in Section \ref{connected components section}, it is the generating set (as an abstract group) of $G^\infty(\mathfrak{C})$, and hence equal to $G^\infty(\mathfrak{C})$; in particular it is contained in $G^{00}(\mathfrak{C})$. On the other hand, it is quickly checked to be type-definable, and is hence a type-definable subgroup of bounded index, and so contains $G^{00}(\mathfrak{C})$. So it coincides with $G^{00}(\mathfrak{C})$.
\end{proof}

\begin{corollary}\label{f-generic is g00-invariant}
Any global f-generic type is $G^{00}(\mathfrak{C})$-invariant.
\end{corollary}
\begin{proof}
    Suppose otherwise that $p(x)\in S_G(\mathfrak{C})$ is f-generic but not $G^{00}(\mathfrak{C})$-invariant. Then there is a $\mathfrak{C}$-definable set $D\subseteq G$ and an element $g\in G^{00}(\mathfrak{C})$ with $p(x)$ concentrated on $E:=D\setminus gD$. Let $N$ be any small model over which $D$ and $g$ are both defined. By Corollary \ref{g00=g000}, there are $a,b\in G(\mathfrak{C})$ with $a\equiv_N b$ and $g=a^{-1}b$. Since $E$ is defined over $N$ and $p(x)\vdash x\in E$, by Lemma \ref{chain condition} we have that $E\wedge a^{-1}bE$, i.e. $E\wedge gE$, is f-generic. But this set is contained in $gD\setminus gD=\varnothing$, a contradiction.
\end{proof} Now from Corollary \ref{f-generic is g00-invariant} and Fact \ref{countable case} we obtain the main theorem; it is a completely standard consequence, but we give details just for completeness.
\begin{theorem}\label{DA main theorem}
    Suppose $G$ admits a global f-generic type. Then $G$ is definably amenable.
\end{theorem}
\begin{proof}
    Note that any reduct of an NIP theory is still NIP, and, by the `$G$-dividing' characterization of f-genericity in Fact \ref{G-dividing characterization}, any reduct of an f-generic type to a language over which $G$ is still defined will remain f-generic. In particular, by Fact \ref{countable case}, if $G$ admits a global f-generic type then the reduct of $G$ to any countable sublanguage of $L$ over which $G$ remains defined is definably amenable.

    The main result follows now from standard arguments, using compactness in the space of Keisler measure on $G$. More precisely, the space $\mathfrak{M}_G(\mathfrak{C})$ of Keisler measures on $G(\mathfrak{C})$ is a closed subspace of the compact space of $[0,1]$-valued functions on the Boolean algebra of $L(\mathfrak{C})$-definable subsets of $G(\mathfrak{C})$, equipped with the topology of pointwise converge. So a basic open set of $\mathfrak{M}_G(\mathfrak{C})$ is of form $\{\mu:r_1<\mu(\phi(x,b))<r_2\}$, where $r_1,r_2\in [0,1]$ and $\phi(x,b)$ is an $L(\mathfrak{C})$-definable subset of $G$.

    Now the set $X\subseteq\mathfrak{M}_G(\mathfrak{C})$ of $G(\mathfrak{C})$-invariant measures is a closed subspace, given by the intersection of all sets of form $\{\mu:\mu(\phi(x,b)\triangle g\phi(x,b))=0\}$, where $\phi(x,b)$ is an $L(\mathfrak{C})$-definable subset of $G(\mathfrak{C})$ and $g\in G(\mathfrak{C})$. Definable amenability of $G$ is equivalent to non-emptyness of $X$, and so, since $\mathfrak{M}_G(\mathfrak{C})$ is compact, it suffices to show that for any finitely many formulas $\phi_i(x,b_i),i\leqslant n$ and finitely many group elements $g_i\in G(\mathfrak{C}),i\leqslant n$ we can find some $\mu\in\mathfrak{M}_x(\mathfrak{C})$ assigning measure $0$ to all the formulas $\phi_i(x,b_i)\triangle g_i\phi_i(x,b_i)$. By the first paragraph, letting $L_0$ be a countable sublanguage of $L$ over which $G$ and the $\phi_i$ are still defined, we can find a $G(\mathfrak{C})$-invariant Keisler measure $\mu_0$ on the Boolean algebra of $L_0(\mathfrak{C})$-definable subsets of $G(\mathfrak{C})$; in particular $\mu_0$ assigns measure $0$ to each formula $\phi_i(x,b_i)\triangle g_i\phi_i(x,b_i)$. Now on general grounds we can extend $\mu_0$ to a Keisler measure $\mu$ on the Boolean algebra of $L(\mathfrak{C})$-definable subsets of $G(\mathfrak{C})$, giving the desired result; see \cite{los}. 
\end{proof}

By results from \cite{hrushovski_pillay} and \cite{chernikov_simon}, we obtain the following:

\begin{corollary}\label{DA main corollary}
    The following are equivalent for a group $G$ definable in an NIP theory: \begin{enumerate}
        \item $G$ is definably amenable.
        \item $G$ admits a global type invariant under left-translation by elements of $G^{00}$.
        \item $G$ admits a global f-generic type.
        \item The non-f-generic, equivalently $G$-dividing, definable subsets of $G$ form an ideal.
    \end{enumerate}
\end{corollary}
\begin{proof}
    That 1 implies 2 is from \cite{hrushovski_pillay}, and that 2 implies 3 is from \cite{chernikov_simon}. The equivalence of 3 and 4 is Corollary \ref{global type iff ideal}, and that 3 implies 1 is Theorem \ref{DA main theorem}.
\end{proof}

\section{An application to dp-minimal groups}\label{dp-minimal section}
In this section, we do \textit{not} assume that $T$ is NIP. Recall that a theory is `inp-minimal' if there do not exist formulas $\phi(x,y)$ and $\psi(x,z)$, where $x$ is a singleton variable of the home sort, and indiscernible sequences $(b_i)_{i\in\omega}$ and $(c_i)_{i\in\omega}$, such that $\phi(x,b_i)\wedge\psi(x,c_j)$ is consistent for each $i,j\in\omega$ but $\{\phi(x,b_i):i\in\omega\}$ and $\{\psi(x,c_i):i\in\omega\}$ are each inconsistent.

A theory is `dp-minimal' if it is both inp-minimal and NIP. Dp-minimality was first studied in \cite{shelah_forking} and \cite{shelah_strong}, and then isolated as a notion in \cite{onshuus_usvyatsov}; see also for example \cite{dolich_goodrick_lippel} for an introduction to the notion.

In Problem 5.9 of \cite{chernikov_pillay_simon} and Problem 3.13 of \cite{kaplan_levi_simon}, it was asked whether every dp-minimal group is definably amenable. We show here that a positive answer follows quickly from Theorem \ref{DA main theorem}. First we need Lemma \ref{dp-minimal preliminary lemma}, which is a slightly more general version of Proposition \ref{universal witnesses} and is proved in the same way. We state it in terms of non-$G$-dividing rather than f-genericity, in case it is useful in contexts where forking and dividing do not coincide over models.

\begin{lemma}\label{dp-minimal preliminary lemma}
    Suppose that $D,E\subseteq G$ are definable sets, that $D$ $G$-divides, and that $D\vee E$ does not $G$-divide. Then there is a $k\in\omega$ and a sequence $(e_i:i\in\omega)$ from $E$ such that the family of translates $(e_i^{-1}D:i\in\omega)$ is $k$-inconsistent.
\end{lemma}
\begin{proof}
    Fix some small model $M$ over which $D,E$ are defined. Since $D$ $G$-divides, there is a sequence $(g_i:i\in\omega)$ of elements of $G$ and a $k\in\omega$ such that the family of translates $(g_iD:i\in\omega)$ is $k$-inconsistent; by `Ramsey and compactness', i.e. replacing $(g_i:i\in\omega)$ by an $M$-indiscernible sequence realizing its EM-type over $M$, we may assume that $(g_i:i\in\omega)$ is $M$-indiscernible. Since $D\vee E$ is an $M$-definable set, and does not $G$-divide, and since $(g_i:i\in\omega)$ is $M$-indiscernible,     we have $\bigwedge_{i\in\omega}g_i(D\vee E)\neq\varnothing$, which forces $\bigwedge_{i\in\omega}g_iE\neq\varnothing$ by pigeonhole and indiscernibility; let $a\in\bigwedge_{i\in\omega}g_iE$. Now letting $e_i:=g_i^{-1}a$, then $e_i\in E$ for every $i\in\omega$. On the other hand, the translates $e_i^{-1}D$ are $k$-inconsistent, since for any $s\subset\omega$ of size $k$ we have $\bigwedge_{i\in s}e_i^{-1}D=a^{-1}\bigwedge_{i\in s}g_iD=\varnothing$.
\end{proof}

\begin{lemma}\label{inp-minimal ideal}
    If $G$ is inp-minimal, then the $G$-dividing definable subsets of $G$ form an ideal.
\end{lemma}
\begin{proof}
    Suppose otherwise. Then there are some definable $D,E\subseteq G$ that each $G$-divide but such that $D\vee E$ does not $G$-divide. By the previous lemma (which applies symmetrically to $D$ and to $E$), we may find sequences $(e_i:i\in\omega)$ from $E$ and $(d_i:i\in\omega)$ from $D$ such that each of the families $(e_i^{-1}D:i\in\omega)$ and $(d_i^{-1}E:i\in\omega)$ is $k$-inconsistent for some $k\in\omega$. Fixing a small model $M$ over which $D,E$ are defined, we may by `Ramsey and compactness' assume that the sequence $(d_i:i\in\omega)$ and $(e_i:i\in\omega)$ are each $M$-indiscernible. (Since $D,E$ are $M$-definable, the EM-type of the sequences over $M$ contains the information of lying in $D,E$, respectively, and likewise the EM-type contains the information of $k$-inconsistency of the respective families of translates.) Since the family $(d_i^{-1}E:i\in\omega)$ is $k$-inconsistent, taking inverses gives that the family $(E^{-1}d_i:i\in\omega)$ is also $k$-inconsistent. On the other hand, for any $i,j\in\omega$, we have $e_i\in E$ and $d_j\in D$, so that $e_i^{-1}D\wedge E^{-1}d_j$ contains $e_i^{-1}d_j$ and is hence consistent. This contradicts inp-minimality.
\end{proof}

Forking and dividing coincide over models in inp-minimal theories, so for the same reasons as in NIP theories Fact \ref{G-dividing characterization} holds for inp-minimal theories. In particular Theorem \ref{DA main theorem} and Lemma \ref{inp-minimal ideal} now give the desired result:
\begin{theorem}\label{dp-minimal main theorem}
    If $G$ is dp-minimal then it is definably amenable.
\end{theorem}
\begin{proof}
    Since $G$ is inp-minimal, by Lemma \ref{inp-minimal ideal} and the remark above this theorem $G$ admits global f-generic types. Since $G$ is NIP, the claim follows from Theorem \ref{DA main theorem}.
\end{proof}

It is natural to ask the analogous question about inp-minimal groups. This is unclear, even in the case of simple theories. On the one hand, groups definable in simple theories \textit{always} admit (strong) f-generic types, by the results of \cite{pillay_simple}. On the other hand, there exist groups definable in simple theories (even in supersimple theories of SU-rank 1) that are not definably amenable, as demonstrated in \cite{seven_authors}. So even in the most well-behaved examples outside NIP, the connection between definable amenability and good behavior of f-genericity fails. This however does not rule out the possibility of inp-minimal groups always being definably amenable for other reasons; the example in \cite{seven_authors}, despite being definable in an inp-minimal theory, is not itself inp-minimal, and to our knowledge it is the only known example of a group definable in a simple theory that is not definably amenable. However, any attempt to understand definable amenability of inp-minimal groups would need to use very different techniques than the machinery of f-generic types given in this paper.

\section{The type-definable case}\label{type-definable section}
In this section we will deal with the case of a type-definable group in an NIP theory. So, throughout the section, assume that $T$ is NIP and that $(G,\cdot)$ is a type-definable group. This means that $G(x)$ is a partial type over a small set of parameters, and that there are definable supersets $X,Y,Z\supseteq G$ and a definable map $\cdot:X\times Y\to Z$ such that $(G(\mathfrak{C}),\cdot)$ is a group. Replacing $\cdot$ by the map taking $(a,b)$ to $a\cdot b$ if $a\in X$ and $b\in Y$ and to $a$ otherwise we may assume without loss that $\cdot$ is a total binary operation $\mathfrak{C}^x\times\mathfrak{C}^x\to\mathfrak{C}^x$. (This is purely for notational convenience.)

Most of the general theory of NIP groups still goes through in the type-definable case. For example, the connected components $G^{00}$ and $G^{\infty}$ exist, and are defined in exactly the same way as in the definable case. Definable amenability can also be defined for the type-definable case: $G(x)$ is definably amenable if there is a Keisler measure $\mu_x\in\mathfrak{M}_x(\mathfrak{C})$ such that $\mu(\phi(x,b))=1$ for every formula $\phi(x,b)$ with $G(x)\vdash\phi(x,b)$ and such that $\mu(\phi(gx,b))=\mu(\phi(x,b))$ for every $g\in G(\mathfrak{C})$ and every $L(\mathfrak{C})$-formula $\phi(x,b)$.

Moreover, the analogue of Fact \ref{countable case} still applies, by the same argument discussed in Section \ref{connected components section}:\begin{fact}\label{g00 invariant implies DA type-definable}
    Suppose $T$ is countable and NIP and that $G$ is type-definable over a countable parameter set. If there is a global $G^{00}$-invariant type concentrated on $G$, then $G$ is definably amenable.
\end{fact}

Now, given a formula $\psi(x,c)$, let us say that $\psi(x,c)$ is `f-generic (for $G$)' if, for every $g\in G(\mathfrak{C})$, and for some (every) small model $M\prec\mathfrak{C}$ over which $G$ and $\psi(x,c)$ are defined, the partial type $G(x)\wedge\psi(gx,c)$ does not fork over $M$. By Fact \ref{forking=dividing}, $\psi(x,c)$ is f-generic for $G$ if and only if for some (every) $M\prec\mathfrak{C}$ over which $G$ and $\psi(x,c)$ are defined, and for every $M$-indiscernible sequence $(g_i)_{i\in\omega}$ of elements of $G(\mathfrak{C})$, the partial type $G(x)\wedge\bigwedge_{i\in\omega}\psi(g_ix,c)$ is inconsistent.

Let us further say that $\psi(x,c)$ `$G$-divides' if there is a definable set $\phi(x,b)$ in the partial type defining $G(x)$ and a 
sequence of elements $(g_i:i\in\omega)$ of $G(\mathfrak{C})$ such that, for some $k\in\omega$, the formula $\phi(x,b)\wedge\bigwedge_{i\in s}\psi(g_ix,c)$ is inconsistent for every $s\subset\omega$ of size $k$. (As usual, we will say that a partial type $\pi(x)$ $G$-divides if and only if it implies some $G$-dividing formula.) Now by a Ramsey+compactness argument, non-$G$-dividing is equivalent to f-genericity for $G$.

With these new definitions in place, all of the proofs from Section \ref{general facts subsection}, and all of the proofs from Section \ref{results section} up until and including Corollary \ref{f-generic is g00-invariant}, go through essentially without modification. The only change is that, instead of considering definable subsets $D\subseteq G$, we would consider `relatively' definable subsets of $G$, namely partial types of form $\phi(x,b)\wedge G(x)$ for some $L(\mathfrak{C})$-formula $\phi(x,b)$. Summarizing, we obtain the following:


\begin{proposition}\label{f-gen implies g00-inv type-definable}
    A global type concentrated on $G$ that does not $G$-divide is $G^{00}$-invariant.
\end{proposition}

Now we would like to argue as in Theorem \ref{DA main theorem} to deduce the main result from Fact \ref{g00 invariant implies DA type-definable} and Proposition \ref{f-gen implies g00-inv type-definable}. However, there is a bit of subtlety. If $G$ is type-definable by an intersection of countably many formulas, then we may indeed reduce to the case of a countable reduct and a countable parameter set, and we will be done. But $G$ may not be type-definable by an intersection of countably many formulas, and in that case the argument is more involved. To work around the issue, we will run the usual argument that $G$ is an intersection of type-definable groups each defined by an intersection of countably formulas; however, we need to guarantee that these type-definable groups themselves admit f-generic types, and for that we need the following somewhat technical lemma:

\begin{lemma}\label{technical lemma type-definable}
    Suppose $G(x)$ is a partial type without parameters such that $\cdot$ defines a group operation on $G$, and suppose there exists a global type concentrated on $G$ that does not $G$-divide. Let $\eta(x)$ be a formula without parameters such that $G(x)\vdash \eta(x)$, and let $L_0$ be a countable sublanguage of $L$ such that $\eta(x)$ and $\cdot$ are both defined in $L_0$. Then there is a countable sublanguage $L_\eta\subseteq L$, containing $L_0$, and a partial type $G_\eta(x)$, type-definable without parameters by formulas from $L_\eta$ and such that $\cdot$ defines a group operation on $G_\eta$, such that (i) $G(x)\vdash G_\eta(x)$ (ii) $G_\eta(x)\vdash \eta(x)$, and (iii) there is a global $L_\eta$-type concentrated on $G_\eta$ that does not $G_\eta$-divide.
\end{lemma}

\begin{proof}
    We will do an `interleaving' argument, inductively constructing a sequence of countable sublanguages $L_0\subseteq L_1\subseteq\dots$ of $L$. For each $t\in\omega$, we will define $Q_t(x)$ to be the partial type consisting of all $L_t$-formulas $\phi(x)$ without parameters such that $G(x)\vdash \phi(x)$. So each $Q_t$ will be countably defined without parameters, and for each $t\in\omega$ we will have $G(x)\vdash Q_t(x)$ and $Q_{t+1}(x)\vdash Q_t(x)$. To conclude we will take $G_\eta(x)=\bigwedge_{t\in\omega}Q_t(x)$ and $L_\eta=\bigcup_{t\in\omega}L_t$.
    
    For the base step, let $L_0$ be the language given in the theorem statement. Note in particular that $Q_0(x)\vdash\eta(x)$. Now suppose that we have constructed $L_t$ and hence $Q_t$. We split into two cases.
    
    First suppose $t$ is even. Let $\phi(x)$ be a formula in the partial type defining $Q_t$, i.e. an $L_t$-formula without parameters such that $G(x)\vdash \phi(x)$. By compactness, we may find an $L$-formula $\alpha_\phi(x)$ in the partial type defining $G(x)$ such that (i) the product of any two elements of $\alpha_\phi(x)$ satisfies $\phi(x)$, and (ii) every element of $\alpha_\phi(x)$ has a two-sided inverse in $\phi(x)$. Let $L_{t+1}$ be any countable sublanguage of $L$ containing $L_t$ and over which all of the $\alpha_\phi$ are defined. Then $ab\models Q_t(x)$ for all $a,b\models Q_{t+1}(x)$ and every element of $Q_{t+1}(x)$ has a two-sided inverse in $Q_t(x)$.

    Now suppose $t$ is odd. Let $\phi(x)$ be a formula in the partial type defining $Q_t$, let $\psi_1(x,y_1),\dots,\psi_n(x,y_n)$ be a finite collection of formulas of $L_t$, without parameters, and let $k\in\omega$ be a natural number. For each $i\in [n]$, let $\sigma_i(y_i)$ be the partial type $$\exists u_j:j\in\omega\left[\bigwedge_{j\in\omega}G(u_j)\wedge\bigwedge_{s\subset\omega,|s|=k}\neg\exists x\left(\phi(x)\wedge\bigwedge_{j\in s}\psi_i(u_jx,y_i)\right)\right].$$ Now let $\pi(y_1,\dots,y_n)$ be the partial type which contains $\sigma_1(y_1)\wedge\dots\wedge\sigma_n(y_n)$ and $\forall x\left[\phi(x)\to\bigvee_{i\in[n] }\psi_i(x,y_i)\right]$. By the assumption that there is a global type concentrated on $G$ that does not $G$-divide, $\pi(y_1,\dots,y_n)$ is inconsistent; indeed if $b_i\models\sigma_i(y_i)$, then $\psi_i(x,b_i)$ $G$-divides, and if additionally $(b_1,\dots,b_n)$ realizes $\forall x\left[\phi(x)\to\bigvee_{i\in[n] }\psi_i(x,y_i)\right]$ then every global type concentrated on $G$ must contain some $\psi_i(x,b_i)$. So, by compactness, there is some $L$-formula $\beta_{\phi,\psi_1,\dots,\psi_n,k}(x)$ in the partial type defining $G(x)$ such that, replacing the $G(u_j)$ in $\sigma_i(y_i)$ by $\beta_{\phi,\psi_1,\dots,\psi_n,k}(u_j)$, one still gets inconsistency in $\pi(y_1,\dots,y_n)$. Let $L_{t+1}$ be any countable sublanguage of $L$ containing $L_t$ and over which all of the $\beta_{\phi,\psi_1,\dots,\psi_n,k}$ are defined.
    
    As described in the first paragraph, let $L_\eta=\bigcup_{t\in\omega}L_t$ and $G_\eta(x)=\bigwedge_{t\in\omega}Q_t(x)$. Since $Q_0(x)\vdash\eta(x)$, $G_\eta(x)\vdash\eta(x)$. By definition of the $Q_t$, $G(x)\vdash G_\eta(x)$ and $G_\eta(x)$ is type-defined by $L_\eta$-formulas without parameters. By the even stages of the construction, $(G_\eta,\cdot)$ is a group. So we need to show that there is a global $L_\eta$-type concentrated on $G_\eta$ that does not $G_\eta$-divide.

    It is enough to show that, for any $L_\eta$-formula $\phi_0(x)$ in the partial type defining $G_\eta(x)$, and any $L_\eta(\mathfrak{C})$-formulas $\psi_1(x,b_1),\dots,\psi_n(x,b_n)$ that all $G_\eta$-divide, the formula $\phi_0(x)\wedge\bigwedge_{i\in[n]}\neg\psi_i(x,b_i)$ is consistent. Suppose otherwise. Then $\phi_0(x)\vdash\bigvee_{i\in[n]}\psi_i(x,b_i)$. Also, since $\psi_i(x,b_i)$ $G_\eta$-divides, there is (by definition) a formula $\phi_i(x)$ in the partial type defining $G_\eta$ and a sequence $(g_{ij}:j\in\omega)$ of elements of $G_\eta(\mathfrak{C})$ such that, for some $k_i\in\omega$, the formula $\phi_i(x)\wedge\bigwedge_{j\in s}\psi_i(g_{ij}x,b_i)$ is inconsistent for every $s\subset\omega$ of size $k_i$. Now pick $t\in\omega$ an odd number such that $\phi_0(x)$ and the $\phi_i(x)$ and the $\psi_i(x,y_i)$ are all $L_t$-formulas. Letting $\phi(x)=\phi_0(x)\wedge\bigwedge_{i\in[n]}\phi_i(x)$ and $k=\max_{i\in[n]}k_i$, we now get a contradiction to the fact that $Q_{t+1}(x)\vdash\beta_{\phi,\psi_1,\dots,\psi_n,k}(x)$ and hence that each $g_{ij}$ realizes $\beta_{\phi,\psi_1,\dots,\psi_n,k}(x)$.
    \end{proof}
So altogether we get the following.
\begin{theorem}\label{main theorem type-definable}
    If there is a global type concentrated on $G$ that does not $G$-divide, then $G$ is definably amenable.
\end{theorem}
\begin{proof}
    By adding constant symbols for the parameters defining $G$, we may assume that $G(x)$ is type-definable by $L$-formulas without parameters. We will argue as in Theorem \ref{DA main theorem}; in this case we are trying to construct an element of $\mathfrak{M}_x(\mathfrak{C})$ in the intersection of (i) all sets of form $\{\mu:\mu(\theta(x,c)\triangle \theta(gx,c))=0\}$ for $\theta(x,c)$ an $L(\mathfrak{C})$-formula and $g\in G(\mathfrak{C})$, and (ii) all sets of form $\{\mu:\mu(\eta(x))=1\}$ for $\eta(x)$ an $L$-formula without parameters such that $G(x)\vdash\eta(x)$. By compactness in the space $\mathfrak{M}_x(\mathfrak{C})$, we need only show that finitely many conditions of this kind can be satisfied. So, given finitely many conditions, let $L_0$ be a countable sublanguage of $L$ over which all of the formulas in the conditions are defined. Let $\eta(x)$ be the conjunction of all the formulas appearing in the finitely many conditions of type (ii).

    Let $G_\eta$ and $L_\eta$ be given by Lemma \ref{technical lemma type-definable} for $L_0$ and $\eta(x)$. Then there is a global $L_\eta$-type $p(x)$ concentrated on $G_\eta(x)$ that does not $G_\eta$-divide. By the definition of $G_\eta$-dividing, $p(x)$ still does not $G_\eta$-divide in the reduct to $L_\eta$. So, by Proposition \ref{f-gen implies g00-inv type-definable} applied in the reduct, $p(x)$ is invariant under $G^{00}_\eta$ \textit{as computed in the reduct}. So, by Fact \ref{g00 invariant implies DA type-definable} applied in the reduct, $G_\eta$ is definably amenable in the reduct. Any Keisler measure witnessing this will satisfy all of the desired closed conditions, and, as in the proof of Theorem \ref{DA main theorem}, we can extend the Keisler measure from the reduct to a Keisler measure for the original language, giving the desired result.
\end{proof}

So indeed, in an NIP theory, a type-definable group with a global f-generic type is definably amenable. As far as we can tell, Theorem \ref{main theorem type-definable} also gives the first proof that a type-definable group with a global \textit{strongly} f-generic type is definably amenable; we refer back to Section \ref{strong f-genericity subsection} for the thorough discussion of this. So Theorem \ref{main theorem type-definable} has the pleasant consequence of giving a corrected proof for Proposition 5.7 of \cite{hrushovski_pillay}.

\newpage


\end{document}